\documentclass[reqno, 11pt]{amsart}
\usepackage{amsmath,mathtools}
 \usepackage{amssymb}
\usepackage{amsthm}
\usepackage{amsmath}
\usepackage{times}
\usepackage{latexsym}
\usepackage[mathscr]{eucal}
\usepackage{graphicx}
\usepackage{float} 

\numberwithin{equation}{section}
 
  \newtheorem{theorem}{Theorem}[section]
  \newtheorem{proposition}[theorem]{Proposition}
  \newtheorem{lemma}[theorem]{Lemma}
  \newtheorem{corollary}[theorem]{Corollary}

  \newtheorem{definition}[theorem]{Definition}
  \newtheorem{example}[theorem]{Example}

\title[On the geometry of null hypersurfaces]{On the geometry of null  hypersurfaces of indefinite complex contact manifolds}
\author[Samuel Ssekajja]{Samuel Ssekajja*}
\newcommand{\acr}{\newline\indent}
\address{\llap{*\,} School of Mathematics\acr
 University of the Witwatersrand\acr
 Private Bag 3, Wits 2050\acr
South Africa}
\email{samuel.ssekajja@wits.ac.za} 
\thanks{}


\begin{document}
\begin{abstract}
We study the geometry of null hypersurfaces in indefinite complex contact manifolds. We prove several classification results for a variety of well-known null hypersurfaces, including the totally umbilic, totally screen umbilic, and the screen conformal ones. Furthermore, a characterisation of the ambient space is given in case  the underlying null hypersurface is totally contact umbilic, totally contact screen umbilic or contact screen conformal, i.e. we have proved  that the ambient complex contact manifold must be a space of constant $GH$-sectional curvature of $-3$.
\end{abstract}
\maketitle
\section{Introduction} 

On any semi-Riemannian manifold there is a natural existence of null (lightlike) subspaces. In 1996, Duggal-Bejancu published a book \cite{db} on the null  geometry of submanifolds which fiiled an important missing part in the general theory of submanifolds. This book was later updated by Duggal-Sahin in \cite{ds2}, by collecting most of the new discoveries in the area since the first publiccation. Away from these two book, many researchers have investigated the geometry of null subspaces of semi-Riemannian manifolds. On the other hand, in about the same tine as in the book \cite{db}, Kupeli \cite{kup} introduces the theory of null geometry in a relatively different way. The main tool in his approach was the consideration of a factor bundle which is isomorphic to the screen distribution used by the authors in \cite{db}. In Chapter 8 of \cite{ds2}, the authors introduces the geometry of null submanifolds of indefinite quaternion Kaehler manifolds. Therein, the authors study the geometry of real null hypersurfaces, the structure of null submanifolds, both, of indefinite quaternion Kaehler manifolds and show that a quaternion null submanifold is always totally geodesic. This result implies that the study of null submanifolds, other than quaternion null submanifolds, is interesting. Then, they  deal with the geometry of screen real submanifolds in detail. As a generalization of real null hypersurfaces of quaternion Kaehler manifolds, they introduced $QR$-null submanifolds. Furthermaore, they show that the class of $QR$-null submanifolds does not include quaternion null submanifolds and screen real submanifolds. They also  introduced and studied the geometry of screen $QR$-null and screen $CR$-null submanifolds as generalizations of quaternion  null submanifolds and screen real submanifolds, and provided examples for each class of null submanifolds of indefinite quaternion Kaehler manifolds.

Despite all the above contributions, we remark that the null geometry of submanifolds of indefinite Sasakian 3-structure manifolds, as well as indefinite complex contact manifold have not yet been studied. In \cite{bla1}, the geometry of complex contact manifolds in the Riemannian sense is done in which the foundations on such manifiolds is given, from structures to their curvatures. The objective of this paper is to introduce the geometry of null hypersurfaces of indefinite complex contact manifolds. Several characterisation results are proved. In fact, we show that normal indefinite complex contact manifolds do not admit any totally umbilic, totally screen umbilic as well as screen conformal null hypersurfaces tangent to the vertical distribution. The notions of totally contact umbilic, totally contact screen umbilic and contact screen conformal are explored in details. In particular we prove that an indefinite complex contact manifold of constant  $GH$-sectional curvature different from $-3$ does not admit any totally contact umbilic, totally contact screen umbilic and contact screen conformal null hypersurfaces, tangent to the characteristic subbundle.

The paper is arranged as follows; In Section \ref{pre}, we quote some basic notions on complex contact manifolds as well as null hypersurfaces needed in the rest the paper. In Section \ref{main}, we prove several  non-existence results and in Section \ref{mainn}, we focus on contact umbilic, contact screen umbilic and contact screen conformal null hypersurface in indefinite complex contact space forms. 
 

\section{Preliminaries} \label{pre}

A {\it complex contact manifold} is a complex manifold, $\overline{M}$, of odd complex dimension $(2n+1)$ together with an open covering $\{\mathcal{O}_{i}\}$ by coordinate neighbourhoods such that: (1) On each $\mathcal{O}_{i}$ there is a holomorphic 1-form $\theta_{i}$ such that $\theta_{i}\wedge (d\theta_{i})^{n}\ne 0$. (2)  On $\mathcal{O}_{i}\cup \mathcal{O}_{j} \ne \emptyset$ there is a non-vanishing holomorphic function $f_{\alpha\beta}$ such that $\theta_{i}=f_{\ij}\theta_{j}$ (see \cite{bla1,bla2} for more details). Furthermore, the subspaces $\{X\in T_{m}\mathcal{O}_{i}: \theta_{i}(X)=0\}$ defines a non-integrable holomorphic subbundle $\mathcal{H}$ of complex dimension $2n$ called the {\it complex contact subbundle} or {\it horizontal subbundle}. The quotient $L=T\overline{M}/\mathcal{H}$ is a complex line bundle over $\overline{M}$ \cite[p. 49]{bla2}. Some well-known examples of complex contact metric manifolds include the complex Heisenberg group $H_{\mathbb{C}}$ and the odd-dimensional complex projective space, see \cite{bla1, bla2} for more details on these manifolds. Define a local section $U$ of $T\overline{M}$, i.e., a section of $T\mathcal{O}$,  by $du(U,X)=0$,  for every $X\in \mathcal{H}$, $u(U)=1$ and $v(U)=0$. Such local sections then define a global subbundle $\mathcal{V}$ by $\mathcal{V}|_{\mathcal{O}}=\mathrm{Span}\{U,JU\}$. Then, we have $T\overline{M}=\mathcal{H}\perp \mathcal{V}$ and we denote the projection map to $\mathcal{H}$ by $p:T\overline{M}\longrightarrow \mathcal{H}$. The subbundle $\mathcal{V}$ is called the {\it vertical subbundle} or {\it characteristic subbundle}. On the other hand, if $\overline{M}$ is a complex manifold with almost complex structure $J$, Hermitian metric $\overline{g}$ and open covering by coordinate neighbourhoods $\{\mathcal{O}_{i}\}$, $\overline{M}$ is called a {\it complex almost contact metric manifold} if it satisfies the following two conditions: (1) On each $\mathcal{O}_{i}$, there exists 1-forms $u_{i}$, $v_{i}=u_{i}\circ J$, with orthogonal dual vector fields $U_{i}$ and $V_{i}=-JU_{i}$, and (1,1)-tensor fields $G_{i}$ and $H_{i}=G_{i}J$ such that 
\begin{align}
	&\;\;\;\;\;\;\;\;H^{2}_{i}=G^{2}_{i}=-I+u_{i}\otimes U_{i}+v_{i}\otimes V_{i},\label{cm1}\\
	&\overline{g}(G_{i}X,Y)=-\overline{g}(X,G_{i}Y),\;\;\;\; \overline{g}(U_{i},X)=u_{i}(X),\label{cm2}\\
	&\;\;\;\;\;\;\;\;G_{i}J=-JG_{i},\;\;\; G_{i}U=0,\;\;\; u_{i}(U)=1,\label{cm3}
\end{align}
 for all $X,Y\in\Gamma(T\overline{M})$.  (2) On the overlaps $\mathcal{O}_{i}\cap \mathcal{O}_{j}\ne \emptyset$, the above tensors transform as $u_{j}=au_{i}-bv_{i}$, $v_{j}=bu_{i}+av_{i}$, $G_{j}=aG_{i}-bH_{i}$ and  $H_{j}=bG_{i}+aH_{i}$, for some functions $a$, $b$ defined on the overlaps with $a^{2}+b^{2}=1$.
 
 It is obvious that $H_{i}$ also anticommutes with $J$ and is skew-symmetric with respect to $\overline{g}$ and that $G_{i}$ and $H_{i}$ annihilate both $U$ and $V$. Furthermore, the local contact form $\theta$ is $u-iv$ to within a nonvanishing complex-valued function multiple (see \cite{bla1}). Moreover, given a complex contact manifold, a complex almost contact metric structure can be chosen such that 
 \begin{align}\nonumber
 	du(X,Y)&=\overline{g}(X,GY)+(\sigma \wedge v)(X,Y),\\
 	\mbox{and}\;\;\;\;dv(X,Y)&=\overline{g}(X,HY)-(\sigma \wedge u)(X,Y),\nonumber
  \end{align}
 for all $X,Y\in \Gamma(T\overline{M})$, for some 1-form $\sigma$.  In this case we say that $\overline{M}$ has a {\it complex contact metric structure} $(u, v, U, V, G, H, \overline{g})$  \cite{bla1,bla2}.  In this case $\sigma(X)=\overline{g}(\overline{\nabla}_{X}U,V)$, where $\overline{\nabla}$ denotes the Levi-Civita connection on $\overline{M}$. We refer to a complex contact manifold with a complex almost contact metric structure satisfying these conditions as a {\it complex contact metric manifold} \cite{bla1,bla2}.

 Next,  for a complex contact metric structure \cite[p. 237]{bla1} defined local tensor fields $h_{U}$  and $h_{V}$ by 
 \begin{align}\nonumber
 	 h_{U}=\frac{1}{2}\mathrm{sym}(\pounds_{U}G)\circ p\;\;\mbox{and}\;\;\; h_{V}=\frac{1}{2}\mathrm{sym}(\pounds_{V}H)\circ p,
 \end{align}
 where $\mathrm{sym}$ denotes the symmetric part; $h_{U}$ anticommutes with $G$, $h_{V}$ anticommutes with $H$, and 
 \begin{align}
 	&\overline{\nabla}_{X}U=-GX-Gh_{U}X+\sigma(X)V,\label{cm4}\\
 	\mbox{and}\;\;\;&\overline{\nabla}_{X}V=-HX-Hh_{V}X-\sigma(X)U.\label{cm5}
  \end{align}
In view of (\ref{cm4}) and (\ref{cm5}) one easily sees that the integral surfaces of $\mathcal{V}$ are totally geodesic submanifolds. Furthermore, the associated metric $\overline{g}$ is projectable with respect to the foliation induced by the integrable subbundle $\mathcal{V}$ if and only if $h_{U}$ and $h_{V}$ vanish (see \cite{bla1} for more details). Consider the  tensor fields $S$ and $T$ given by 
\begin{align}
	S&(X,Y)=[G,G](X,Y)+2\overline{g}(X,GY)U-2\overline{g}(X,HY)V\nonumber\\
	&\;\;\;\;+2\{v(Y)HX-v(X)HY\}+\sigma(GY)HX-\sigma(GX)HY\nonumber\\
	&\;\;\;\;\;\;\;\;\;\;\;\;\;\;\;\;\;\;\;\;+\sigma(X)GHY-\sigma(Y)GHX,\nonumber\\
	\mbox{and}\;\;\; T&(X,Y)=[H,H](X,Y)-2\overline{g}(X,GY)U+2\overline{g}(X,HY)V\nonumber\\
	&\;\;\;\;+2\{u(Y)GX-u(X)GY\}+\sigma(HX)GY-\sigma(HY)GX\nonumber\\
	&\;\;\;\;\;\;\;\;\;\;\;\;\;\;\;\;\;\;\;\;+\sigma(X)GHY-\sigma(Y)GHX,\nonumber
\end{align} 
for all $X,Y\in \Gamma(T\overline{M})$. In the above, $[G,G]$ and $[H,H]$ denotes the Nijenhuis tensors of $G$ and $H$, respectively. Then, a complex contact metric structure is {\it normal} \cite[p. 251]{bla1} if  $S(X,Y)=T(X,Y)=0$, for all $X,Y\in \Gamma(\mathcal{H})$ and $S(U,X)=T(V,X)=0$, for all $X\in \Gamma(T\overline{M})$. {\it An important consequence of normality is that $h_{U}=0$, for every $U\in \Gamma(\mathcal{V})$}, see \cite[p. 251]{bla1}. Moreover, on  a normal complex contact manifold, $\overline{\nabla}J$, $\overline{\nabla}G$ and $\overline{\nabla}H$ satisfies the relations (see \cite[p. 252]{bla1} for more details). 
\begin{align}
\overline{g}((\overline{\nabla}_{X}J)Y,Z)&=u(X)\nonumber\{d\sigma(Z,GY)-2\overline{g}(HY,Z)\}\nonumber\\
&\;\;\;\;\;\;\;\;\;\;\;\;\;\;\;\;\;\;+v(X)\{d\sigma(Z,HY)+ \overline{g}(GY,Z)\},\label{m1}\\
\overline{g}((\overline{\nabla}_{X}G)Y,Z)&=\sigma(X)\overline{g}(HY,Z)+v(X)d\sigma(GZ,GY)\nonumber\\
	&\;\;\;\;-2v(X)\overline{g}(HGY,Z)-u(Y)\overline{g}(X,Z) -v(Y)\overline{g}(JY,Z)\nonumber\\
	&\;\;\;\;\;\;\;\;\;\;\;\;\;\;\;\;\;\;\;\;+u(Z)\overline{g}(X,Y)+v(Z)\overline{g}(JX,Y),\label{e20}\\
	\overline{g}((\overline{\nabla}_{X}H)Y,Z)&=-\sigma(X)\overline{g}(GY,Z)-u(X)d\sigma(HZ,HY)\nonumber\\
	&\;\;\;\;-2u(X)\overline{g}(GHY,Z)-v(Y)\overline{g}(X,Z) +u(Y)\overline{g}(JY,Z)\nonumber\\
	&\;\;\;\;\;\;\;\;\;\;\;\;\;\;\;\;\;\;\;\;\;\;+v(Z)\overline{g}(X,Y)-u(Z)\overline{g}(JX,Y),\label{e21}
	\end{align}
for all $X,Y,Z\in \Gamma(T\overline{M})$.

For a unit vector $X\in \mathcal{H}_{m}$, the plane in $T_{m}\overline{M}$ spanned by $X$ and $Y=a GX+bHX$, $a,b\in \mathbb{R}$, $a^{2}+b^{2}=1$, is called a $GH$-plane section, and its sectional curvature, $K(X,Y)$, the $GH$-sectional curvature of the plane section. For a given vector $X$, $K(X,Y)$ is independent of the vector $Y$ in the plane of $GX$ and $HX$ if and only if $K(X,GX)=K(X,HX)$ and $\overline{g}(\overline{R}(X,GX)HX,X)=0$. Here, $\overline{R}$ denotes the curvature tensor of $\overline{M}$. Let $\overline{M}$ be a normal complex contact metric manifold; if the $GH$-sectional curvature is independent of the choice of $GH$-section at each point, it is constant on the manifold, and we say that $\overline{M}$ is a {\it complex contact space form} \cite[p. 253]{bla1}. Moreover, the curvature tensor $\overline{R}$ satisfies 
\begin{align}\label{cm6}
	\overline{R}(X,Y)Z&=\frac{c+3}{4}[\overline{g}(Y,Z)X-\overline{g}(X,Z)Y+\overline{g}(Z,JY)JX-\overline{g}(Z,JX)JY	\nonumber\\
	&\;\;\;\;\;\;\;\;\;\;\;\;\; +2\overline{g}(X,JY)JZ]\nonumber\\
	&+\frac{c-1}{4}[-\{u(Y)u(Z)+v(Y)v(Z)\}X+\{u(X)u(Z)\nonumber\\
	&\;\;\;\;\;\;\;\;\;\;\;\;\;+v(X)v(Z)\}Y+2u\wedge v(Z,Y)JX-2u\wedge v(Z,X)JY\nonumber\\
	&\;\;\;\;\;\;\;\;\;\;\;\;\; +4u\wedge v(X,Y)JZ+\overline{g}(Z,GY)GX-\overline{g}(Z,GX)GY\nonumber\\
	&\;\;\;\;\;\;\;\;\;\;\;\;\;+2\overline{g}(X,GY)GZ+ \overline{g}(Z,HY)HX-\overline{g}(Z,HX)HY\nonumber\\
	&\;\;\;\;\;\;\;\;\;\;\;\;\;+2\overline{g}(X,HY)HZ+\{-u(X)\overline{g}(Y,Z)+u(Y)\overline{g}(X,Z)\nonumber\\
	&\;\;\;\;\;\;\;\;\;\;\;\;\;+v(X)\overline{g}(JY,Z)-v(Y)\overline{g}(JX,Z)+2v(Z)\overline{g}(X,JY)\}U\nonumber\\
	&\;\;\;\;\;\;\;\;\;\;\;\;\; +\{-v(X)\overline{g}(Y,Z)+v(Y)\overline{g}(X,Z)-u(X)\overline{g}(JY,Z)\nonumber\\
	& \;\;\;\;\;\;\;\;\;\;\;\;\;+u(Y)\overline{g}(JX,Z)-2u(Z)\overline{g}(X,JY)\}V]\nonumber\\
	&-\frac{3}{4}(d\sigma(U,V)+c+1)[\{v(X)u\wedge v(Z,Y)-v(Y)u\wedge v(Z,X)\nonumber\\
	&\;\;\;\;\;\;\;\;\;\;\;\;\; +2v(Z)u\wedge v(X,Y)\}U-\{u(X)u\wedge v(Z,Y)\nonumber\\
	& \;\;\;\;\;\;\;\;\;\;\;\;\;-u(Y)u\wedge v(Z,X)+2u(Z)v(X,Y)\}V],
\end{align} 
for all $X,Y,Z\in \Gamma(T\overline{M})$.

Let $(\overline{M},\overline{g})$ be a $(m+2)$-dimensional semi-Riemannian manifold with index $q$, where  $0< q < (m+2)$, and consider a null hypersurface $(M,g)$ of $\overline{M}$. Let $g$ be the induced tensor field by $\overline{g}$ on $M$. Then, $M$ is called a \textit{null hypersurface} of $\overline{M}$ if $g$ is of constant rank $m$ and the normal bundle $TM^{\perp}$ is a distribution of rank 1 on $M$ \cite{db}. Here,  the fibres of the vector bundle $TM^{\perp}$ are defined as $T_{x}M^{\perp}=\{Y_{x}\in T_{x}\overline{M}:\overline{g}_{x}(X_{x},Y_{x})=0,\;\forall\, X_{x}\in T_{x}M\}$, for any $x\in M$.  Let $M$  be a null hypersurface, and consider the complementary distribution $S(TM)$ to $TM^{\perp}$ in $TM$, which is called a \textit{screen distribution} \cite{db}. It is well-known that $S(TM)$ is non-degenerate (see \cite{db}). Therefore, we have the decomposition
\begin{align}
	TM=S(TM)\perp TM^{\perp}.
\end{align}
As $S(TM)$ is non-degenerate with respect to $\overline{g}$, we have $T\overline{M}=S(TM)\perp S(TM)^{\perp}$,
where $S(TM)^{\perp}$ is the complementary vector bundle to $S(TM)$ in $T\overline{M}|_{M}$. Let $(M,g)$ be a null hypersurface of $(\overline{M},\overline{g})$ . Then, there exists a unique vector bundle $\mathrm{tr}(TM)$, called the \textit{null transversal bundle} \cite{db} of $M$ with respect to $S(TM)$,  of rank 1 over $M$ such that for any non-zero section $\xi$ of $TM^{\perp}$ on a coordinate neighbourhood $\mathcal{U}\subset M$, there exists a unique section $N$ of $\mathrm{tr}(TM)$ on $\mathcal{U}$ satisfying 
	$\overline{g}(\xi,N)=1$, $\overline{g}(N,N)=\overline{g}(N,Z)=0$, for all  $Z\in \Gamma(S(TM))$. Consequently, we have the following decomposition of $T\overline{M}$.  
\begin{align}\nonumber
	T\overline{M}|_{M}=S(TM)\perp \{TM^{\perp}\oplus \mathrm{tr}(TM)\}=TM\oplus tr(TM).
\end{align}
  Let $\nabla$ and $\nabla^{*}$ denote the induced connections on $M$ and $S(TM)$, respectively, and $P$ be the projection of $TM$ onto $S(TM)$, then the local Gauss-Weingarten equations of $M$ and $S(TM)$ are the following \cite{db}.
\begin{align}
 \overline{\nabla}_{X}Y&=\nabla_{X}Y+B(X,Y)N,\;\;\;\;\;\overline{\nabla}_{X}N=-A_{N}X+\tau(X)N,\label{flow6}\\
  \nabla_{X}PY&= \nabla^{*}_{X}PY + C(X,PY)\xi,\;\;\;\nabla_{X}\xi =-A^{*}_{\xi}X -\tau(X) \xi,\label{flow7}
 \end{align}
 for all $X,Y\in\Gamma(TM)$, $\xi\in\Gamma(TM^{\perp})$ and $N\in\Gamma(\mathrm{tr}(TM))$, where $\overline{\nabla}$ is the Levi-Civita connection on $\overline{M}$. In the above setting, $B$ is the local second fundamental form of $M$ and $C$  is the local second fundamental form on $S(TM)$.  $A_{N}$ and $A^{*}_{\xi}$ are the shape operators on $TM$ and $S(TM)$ respectively, while $\tau$ is a 1-form on $TM$. The above shape operators are related to their local fundamental forms by 
$g(A^{*}_{\xi}X,Y) =B(X,Y)$, $g(A_{N}X,PY) = C(X,PY)$, for any $X,Y\in \Gamma(TM)$. It follows easily that $B(X,\xi)=0$, for all  $X\in \Gamma(TM)$. Moreover, we have $\overline{g}(A^{*}_{\xi}X,N)=0$, $\overline{g}(A_{N} X,N)=0$, for all $ X\in\Gamma(TM)$. Thus,  we immediately  notice that $A_{\xi}^{*}$ and $A_{N}$ are both screen-valued operators.  Let $\vartheta=\overline{g}(N,\boldsymbol{\cdot})$ be a 1-form metrically equivalent to $N$ defined on $\overline{M}$. Take $\eta=i^{*}\vartheta$
to be its restriction on $M$, where $i:M\rightarrow \overline{M}$ is the inclusion map. Then it is easy to show that 
\begin{align}\label{p40}
	(\nabla_{X}g)(Y,Z)=B(X,Y)\eta(Z)+B(X,Z)\eta(Y), 
\end{align}
for all $X,Y,Z\in \Gamma(TM)$. Consequently,  $\nabla$ is generally \textit {not} a metric connection with respect to $g$. However, the induced connection $\nabla^{*}$ on $S(TM)$ is a metric connection. Denote by $\overline{R}$ and  $R$ the curvature tensors of the connection $\overline{\nabla}$ on $\overline{M}$ and the induced linear connections $\nabla$, respectively. Using the Gauss-Weingarten formulae, we obtain the following  Gauss-Codazzi equations for $M$ and $S(TM)$ (see details in \cite{db,ds2}).
\begin{align}
\overline{g}(\overline{R}(X,Y)Z,\xi)=&(\nabla_{X}B)(Y,Z)-(\nabla_{Y}B)(X,Z)\nonumber\\
&\;\;\;\;\;\;\;\;\;\;\;+\tau(X)B(Y,Z)-\tau(Y)B(X,Z),\label{v34}\\
\overline{g}(\overline{R}(X,Y)PZ,N)=&(\nabla_{X}C)(Y,PZ)-(\nabla_{Y}C)(X,PZ)\nonumber\\
&\;\;\;\;\;\;\;\;\;\;\;\;\;\;\;-\tau(X)C(Y,PZ)+\tau(Y)C(X,PZ),\label{v35}
\end{align}
for all $X,Y,Z\in \Gamma(TM)$, $\xi\in \Gamma(TM^{\perp})$ and $N\in \Gamma(\mathrm{tr}(TM))$, where $\nabla B$ and $\nabla C$ are defined as follows;
\begin{align}
	(\nabla_{X}B)(Y,Z)&=XB(Y,Z)-B(\nabla_{X}Y,Z)-B(Y,\nabla_{X}Z),\label{p100}\\
	(\nabla_{X}C)(Y,PZ)&=XC(Y,PZ)-C(\nabla_{X}Y,PZ)-C(Y,\nabla^{*}_{X}PZ),\label{p101}
\end{align}	
for all $X,Y,Z\in \Gamma(TM)$.
   
\section{Some basic results} \label{main}
Let $\overline{M}:=(\overline{M},u, v, U, V, G, H, \overline{g})$ be a $(4n+2)$-dimensional indefinite complex contact manifold, where $\overline{g}$ is a semi-Riemannian metric of index $4q$; $0<q<n$. Next, we construct an indefinite complex contact structure from an indefinite Sasakian 3-structure. 
\begin{example}
\rm{
	Let $(\tilde{M}^{4n+3},\phi_{i}, \xi_{i},\eta_{i},\overline{g})$, for all $i=1,2,3$, be a 3-structure manifold.  Ishihara and Konishi \cite{isa} proved that if one of the contact structures, say $(\phi_{1},\xi_{1},\eta_{1},\overline{g})$,  of a Riemnnnian manifold $\tilde{M}^{4n+3}$ with a (Sasakian) 3-structure is regular,  the base manifold $\overline{M}$ of the induced fibration is a complex contact manifold.  It is easy to see that the above result is also true for a semi-Riemannian manifold $\tilde{M}^{4n+3}_{4q}$, in which the Riemannian metric $\overline{g}$ is replaced with a semi-Riemannian metric of constant index $4q$, $0<q<n$, with one of its structures regular. In this case, the base manifold becomes an indefinite complex contact manifold. Now, let $\tilde{M}^{4n+3}_{4q}$ be an indefinite Sasakian 3-structure manifold. The indefinite complex contact structure on the base space $\overline{M}_{4q}^{4n+2}$ is constructed as follows.	 Consider the structure $(\phi_{1},\xi_{1},\eta_{1},\overline{g})$ as above. Let $\pi :\tilde{M}_{4q}^{4n+3}\longrightarrow \overline{M}_{4q}^{4n+2}$ be the Boothby-Wang	fibration of $\tilde{M}^{4n+3}_{4q}$ over a cosymplectic manifold $\overline{M}^{4n+2}_{4q}$ of integral class (see \cite{bla1} for details). Denoting the horizontal lift with respect to the principal $S^{1}$ bundle connection 1-form $\eta_{1}$  by $\tilde{\pi}$. Then, $JX=\pi_{*}\phi_{1}\tilde{\pi}X$ and, in the analogous way as in \cite{bla1},  the projected metric form an indefinite Kaehler structure on $\overline{M}_{4q}^{4n+2}$. For a coordinate neighbourhood $\mathcal{U}\subset \overline{M}$  and a local cross section $s$ of $\tilde{M}^{4n+3}_{4q}$ over $\mathcal{U}$, the 1-forms $u$ and $v$ and a tensor field $G$ defined on $\mathcal{U}$ by $u(X)\circ \pi=\eta_{2}(s_{*}X)$, $v(X)\circ \pi=\eta_{3}(s_{*}X)$	and $GX=\pi_{*}[\phi_{2}s_{*}X-\eta_{1}(s_{*}X)\xi_{3}+\eta_{3}(s_{*}X)\xi_{1}]$, define the indefinite complex contact and complex almost contact structures on $\overline{M}_{4q}^{4n+2}$.
	} 
\end{example}
Let us set $J_{1}:=J$, $J_{2}:=G$ and $J_{3}:=H$. Then, we have the following lemma.
\begin{lemma}\label{lemma1}
	$J_{a}$'s, for all $a=1,2,3$, satisfies 
	\begin{align}
		&J_{1}J_{2}=-J_{3},\;\;\;J_{1} J_{3}=-J_{3}J_{1}=J_{2},\label{j1}\\
		J_{3}&J_{2}=-J_{2}J_{3}=J_{1}+u\otimes V-v\otimes U,\label{j2}\\
		J_{2}U=J_{3}&U=J_{3}V=0,\;\;\;\overline{g}(J_{3}X,Y)=-\overline{g}(X,J_{3}Y),\label{j3}
	\end{align}
	for all $X,Y\in \Gamma(T\overline{M})$.
\end{lemma}
\begin{proof}
	Using (\ref{cm3}), we have $J_{1}J_{2}=JG=-GJ=-H=-J_{3}$. On the other hand, $J_{1}J_{3}=JH=-HJ=-J_{3}J_{1}=-HJ=-GJ^{2}=G=J_{2}$. This proves (\ref{j1}). Then, in view of (\ref{cm1}) and (\ref{cm3}), we have $J_{3}J_{2}=HG=GJG=-GGJ=-GH=-J_{2}J_{3}$, and $-J_{2}J_{3}=-GH=-G^{2}J=J-(u\circ J)\otimes U-(v\circ J)\otimes V=J_{1}+u\otimes V-v\otimes U$, proving (\ref{j2}). Note that $J_{2}U=GU=0$, by (\ref{cm3}). Also, $J_{3}U=HU=GJU=-JGU=0$. Furthermore, $J_{3}V=HV=GJV=-GJ^{2}U=GU=0$. Finally, for any $X,Y\in \Gamma(T\overline{M})$, we have $\overline{g}(J_{3}X,Y)=\overline{g}(HX,Y)=\overline{g}(GJX,Y)=-\overline{g}(JX,GY)=\overline{g}(X,JGY)=-\overline{g}(X,GJY)=-\overline{g}(X,HY)=\overline{g}(X,J_{3}Y)$, in which we have used (\ref{cm2}) and (\ref{cm3}), which completes the proof.
\end{proof}
Let $(M,g)$ be a null hypersurface of $\overline{M}$. Then, for each $\xi_{x}\in \Gamma(TM^{\perp})$ at $x\in M$, we have $\overline{g}(\xi_{x},\xi_{x})=0$. This means that $\xi_{x}\in \Gamma(T_{x}M)$. Since $J_{1}$ is a complex structure on $\overline{M}$, we have, from (\ref{cm2}), that $\overline{g}(\xi_{x},J_{a}\xi_{x})=0$, for all $a=1,2,3$. Hence, $J_{a}\xi_{x}$ is tangent to $M$. Thus, $J_{a}(TM^{\perp})$ is a distribution on $M$ of rank 3, such that $J_{a}TM^{\perp}\cap TM^{\perp}=\{0\}$. We can, therefore, choose a screen distribution $S(TM)$ of $M$ such that $J_{a}(TM^{\perp})\subset S(TM)$. As the vector fields $U$ and $V$ are space-like, we note that none of them belongs to $TM^{\perp}$ or $\mathrm{tr}(TM)$. Therefore, we can assume that the vertical distribution $\mathcal{V}=\mathrm{Span}\{U,V\}\subset S(TM)$. Then, we have $\overline{g}(J_{a}N,\xi)=-\overline{g}(N,J_{a}\xi)=0$ and $\overline{g}(N,J_{a}N)=0$, for all $N\in \Gamma(\mathrm{tr}(TM))$. Hence deduce that $J_{a}N$ is tangent to $M$ and belongs to $S(TM)$. We know that $\xi$ and $N$ are null vector fields satisfying $\overline{g}(\xi,N)=1$. Thus, $J_{a}\xi$ and $J_{a}N$  are also null vector fields with $\overline{g}(J_{a}\xi,J_{a}N)=1$, for all $a=1,2,3$. Otherwise, we have $\overline{g}(J_{a}\xi,J_{b}N)=0$, for all $a\ne b$. Hence, $J_{a}TM^{\perp}\oplus J_{a}\mathrm{tr}(TM)$ is a vector subbundle of $S(TM)$ rank of $6$. Then, there exist a non-degenerate distribution $D_{0}$ on $M$ such that 
\begin{align}\label{e10}
	S(TM)= \{D_{1}\oplus D_{2}\}\perp D_{0}\perp\mathcal{V},
\end{align}
where 
\begin{align}
	D_{1}&= J TM^{\perp} \perp G TM^{\perp} \perp H TM^{\perp},\\
\mbox{and}\;\;\; D_{2}&= J\mathrm{tr}(TM) \perp G\mathrm{tr}(TM)\perp H\mathrm{tr}(TM).
\end{align}
We have the following characterisation for the distribution $D_{0}\perp \mathcal{V}$.
\begin{proposition}
	$D_{0}\perp \mathcal{V}$ is invariant with respect to $J_{a}$, for all $a=1,2,3$.
\end{proposition}
\begin{proof}
	Using (\ref{cm2}) and Lemma \ref{lemma1}, we gave $\overline{g}(J_{a}X,Y)=-\overline{g}(X,J_{a}Y)$, $a=1,2,3$, for all $X\in \Gamma(D_{0}\perp \mathcal{V})$ and $Y\in \Gamma(TM)$. Now, for $Y=J_{a}\xi$, the last relation gives $\overline{g}(J_{a}X,J_{a}\xi)=-\overline{g}(X,J_{a}^{2}\xi)=-\overline{g}(X,\xi)=0$, and by Lemma \ref{lemma1}, we have $\overline{g}(J_{a}X,J_{b}\xi)=-\overline{g}(X,J_{a}J_{b}\xi)=\overline{g}(X,J_{c}\xi)=0$, $a\ne b$, for any $Y=J_{a}\xi \in  \Gamma(D_{1})$. Hence, $J_{a}X\perp D_{1}$. Also, we have $\overline{g}(J_{a}X,\xi)=-\overline{g}(X,J_{a}\xi)=0$, which shows that $J_{a}X\perp TM^{\perp}$. On the other hand, using (\ref{cm1}), (\ref{cm2}) and Lemma \ref{lemma1}, we have $\overline{g}(J_{a}X,J_{a}N)=-\overline{g}(X,J_{a}^{2}N)=\overline{g}(X,N)=0$ and $\overline{g}(J_{a}X,J_{b}N)=-\overline{g}(X,J_{a}J_{b}N)=-\overline{g}(X,J_{c}N)=0$, for any $N\in \Gamma(\mathrm{tr}(TM))$. Hence, $J_{a}X\perp \{\{D_{1}\oplus D_{2}\}\perp TM^{\perp}\}$. Finally, we have $\overline{g}(J_{a}X,N)=-\overline{g}(X,J_{a}N)=0$, and hence $J_{a}X\perp \{\{D_{1}\oplus D_{2}\}\perp \{TM^{\perp}\otimes \mathrm{tr}(TM)\}\}$, that is $J_{a}(D_{0}\perp \mathcal{V})=D_{0}\perp \mathcal{V}$, for all $a=1,2,3$, which completes the proof.
\end{proof}
\noindent The decompositions of $TM$ and $T\overline{M}$ becomes; 
\begin{align}
	TM&=TM^{\perp}\perp\{D_{1}\oplus D_{2}\}\perp D_{0}\perp \mathcal{V},\label{00}\\
	\mbox{and}\;\;\;\; T\overline{M}&= \{TM^{\perp}\oplus \mathrm{tr}(TM)\}\perp\{D_{1}\oplus D_{2}\}\perp D_{0}\perp \mathcal{V}.\label{oo}
\end{align}
From the decompositions (\ref{00}) and (\ref{oo}), we have the following result.
\begin{proposition}
	If $(M,g)$ is a null hypersurface of an indefinite complex contact manifold $(\overline{M},\overline{g})$, then $\dim (M)\ge 13$ and $\dim (\overline{M})\ge 14$.
\end{proposition}
\noindent Next, let us set  
\begin{align}\label{e11}
V_{a}=-J_{a}\xi\;\;\;\; \mbox{and} \;\;\;\; U_{a}=J_{a}N,\;\;\; \forall\,a=1,2,3.	
\end{align}
Let us consider the distribution $D=\{TM^{\perp}\perp D_{1}\}\perp D_{0}$, and denote by  $S,R$ the projection morphisms of $TM$ onto $D$ and $D_{2}$, respectively. Then, any $X\in \Gamma(TM)$ can be written as 
\begin{align}\label{n1}
	X=SX+\sum_{a=1}^{3}u_{a}(X)U_{a}+u(X)U+v(X)V, 
\end{align}
where $u_{a}(X)$ are 1-forms on $M$ locally defined by 
\begin{align}\label{n2}
	u_{a}(X)=g(X,V_{a}),\;\;\;\;\forall\, a=1,2,3.
\end{align}
Applying $J_{a}$ to (\ref{n1}) leads to 
\begin{align}\label{n3}
	J_{a}X=\phi_{a}X+u_{a}(X)N+u(X)J_{a}U+v(X)J_{a}V,
\end{align}
where $\phi_{a}X:=JSX$. It follows from (\ref{n3}) that 
\begin{align}\label{n4}
	\phi_{a}U=\phi_{a}V=0,\;\;\;\forall\, a=1,2,3.
\end{align} 
On the other hand, from (\ref{n2}), (\ref{n3}), (\ref{cm1}) and (\ref{cm2}), we have $u_{a}(\phi_{a}X)=g(\phi_{a}X,V_{a})=\overline{g}(J_{a}X,V_{a})-u(X)\overline{g}(J_{a}U,V_{a})-v(X)\overline{g}(J_{a}V,V_{a})=0$, for any $a=1,2,3$. Also, $u(\phi_{a}X)=g(\phi_{a}X,U)=\overline{g}(J_{a}X,U)-v(X)\overline{g}(J_{a}V,U)=-\overline{g}(X,J_{a}U)+v(X)\overline{g}(V,J_{a}U)=0$, for any $a=1,2,3$. In a similar way, we have $v(\phi_{a}X)=0$. Therefore, we have 
\begin{align}\label{n6}
	u_{a}\circ \phi_{a}=0,\;\;\; u\circ \phi_{a}=v\circ \phi_{a}=0,\;\;\forall\, a=1,2,3.
\end{align}
Applying $J$ to (\ref{n3}) and using (\ref{n6}), we have 
\begin{align}\label{n7}
	J_{a}^{2}X=\phi_{a}^{2}X-u_{a}(X)U_{a}+u(X)J_{a}^{2}U+v(X)J_{a}^{2}V.
\end{align}
Now, for $a=1$, we have $J_{1}=J$, and relation (\ref{n7}) gives $J^{2}X=\phi_{1}^{2}X-u_{1}(X)U_{1}+u(X)J^{2}U+v(X)J^{2}V$. Since $J^{2}=-I$,  the previous gives 
\begin{align}\label{n8}
	\phi_{1}^{2}X=-X+u_{1}(X)U_{1}+u(X)U+v(X)V,\;\;u_{1}(U_{1})=1.
\end{align}
On the other hand, when $a=2$, we have $J_{2}=G$. Since $J_{2}U=GU=0$ and $J_{2}V=GV=-GJU=-HU=-J_{3}U=0$ (see Lemma \ref{lemma1}), then, (\ref{n7}) gives $G^{2}X=\phi_{2}^{2}X-u_{2}(X)U_{2}$. Now, applying (\ref{cm1}) to this relation we get 
\begin{align}\label{n9}
	\phi_{2}^{2}X=-X+u_{2}(X)U_{2}+u(X)U+v(X)V,\;\;u_{2}(U_{2})=1.
\end{align}
In a similar way, we have 
\begin{align}\label{n10}
	\phi_{3}^{2}X=-X+u_{3}(X)U_{3}+u(X)U+v(X)V,\;\;u_{3}(U_{3})=1.
\end{align}
Now, for all $a\ne b$, we have 
\begin{align}
	u_{a}(U_{b})&=0,\label{n11}\\
	\mbox{and}\;\;\;\;\ \phi_{a}U_{b}&=J_{a}U_{b}-u_{a}(U_{b})N-u(U_{b})J_{a}U-v(U_{b})J_{a}V\nonumber\\
	&=J_{a}U_{b}=-J_{a}J_{b}N=-J_{c}N=U_{c},\label{n12}
\end{align}
in which we have used (\ref{n2}), (\ref{n3}) and Lemma \ref{lemma1}. Furthermore, using (\ref{n2}), (\ref{n3}) and Lemma \ref{lemma1}, we have 
\begin{align}\label{n13}
	(u_{a}\circ \phi_{b})X=u_{a}(\phi_{b}X)=\overline{g}(J_{b}X,V_{a})=\overline{g}(X,J_{b}J_{a}\xi)=u_{c}(X).
\end{align}
Finally, using (\ref{n3}), (\ref{n6}) and Lemma \ref{lemma1}, we have 
\begin{align}\label{n14}
	(\phi_{a}\circ \phi_{b}-U_{a}\otimes u_{b})X&=J_{c}X-u_{c}(X)N-u(X)J_{c}U-v(X)J_{c}V\nonumber\\
	&=\phi_{c}X,\;\;\;\; \forall\,a\ne b.
\end{align}
Putting all the relations (\ref{n2})--(\ref{n14}) together, we have the following result.
\begin{proposition}\label{prop1}
	On a null hypersurface $(M,g)$, tangent to the characteristic subbundle $\mathcal{V}$, of an indefinite complex contact manifold $\overline{M}$, the following holds
	\begin{align}
		\phi_{a}^{2}=-I&+u_{a}\otimes U_{a}+u\otimes U+v\otimes V,\nonumber\\
		u_{a}(U_{a})&=1,\;\; u(U)=1,\;\;v(V)=1,\nonumber\\
		u_{a}\circ \phi_{a}&=0,\;\;u\circ \phi_{a}=0,\;\;
		v\circ \phi_{a}=0,\nonumber\\
		\phi_{a}U_{a}&=0,\;\;\phi_{a}U=0,\;\;\phi_{a}V=0,\nonumber\\
		&\;\;\;\;\;\;\;\mbox{and}\;\;\forall\; a\ne b;\nonumber\\
		u_{a}(U_{b})&=0,\;\;\;\phi_{a}U_{b}=U_{c},\;\; u_{a}\circ \phi_{b}=u_{c}\nonumber\\
		&\;\;\phi_{a}\circ \phi_{b}=\phi_{c}+u_{b}\otimes U_{a}.\nonumber
	\end{align}
\end{proposition}
\noindent Consider  non-zero vector-valued functions $\omega_{a}$, $a=1,2,3$, on a neighbourhood $\mathcal{U}\subset M$. Let $V_{a}'=\omega_{a}V_{a}$, it follows that $U_{a}'=(1/\omega_{a})U_{a}$,  $u_{a}'=u_{a}\circ\omega_{a}$, $v_{a}'=v_{a}\circ (1/\omega_{a})$. Also, let $\omega_{a}U'=U$, $\omega_{a}V'=V$. Then, $u'=u\circ \omega_{a}$ and $v'=v\circ \omega_{a}$. Let us define  $\phi_{a}'$ by $\omega_{a}\circ \phi'_{a}=\phi_{a}\circ \omega_{a}$. Then, applying $\phi'$ from the right hand side of each side in this relation and using the previous relations, we have 
\begin{align}
	\omega_{a}\circ {\phi_{a}'}^{2}&=\phi_{a}\circ (\omega_{a}\circ \phi_{a}^{'})=\phi_{a}^{2}\circ \omega_{a}\nonumber\\
	&=(-I+u_{a}\otimes U_{a}+u\otimes U+v\otimes V)\circ \omega_{a}\nonumber\\
	&=-\omega_{a}+u_{a}'\otimes U_{a}+u'\otimes U+v'\otimes V\nonumber\\
	&=\omega_{a}\circ (-I+u_{a}'\otimes U_{a}'+u'\otimes U'+v'\otimes V'),\nonumber
\end{align}
from which we get ${\phi_{a}'}^{2}=-I+u_{a}'\otimes U_{a}'+u'\otimes U'+v'\otimes V'$, since $\omega_{a}\ne 0$, for each $a=1,2,3$. Furthermore, $\omega_{a}\circ \phi_{a}'U_{a}'=\phi_{a}\circ \omega_{a}U_{a}'=\phi_{a}U_{a}=0$, which implies that $\phi_{a}'U_{a}'=0$. On the other hand, $\omega_{a}\circ \phi_{a}' U'=\phi_{a}\circ \omega_{a}U'=\phi_{a}U=0$. This implies that $\phi_{a}' U'=0$. In same way, we have $\phi_{a}' V'=0$. Note that $u_{a}'\circ \phi_{a}'=u'\circ \phi_{a}'=v'\circ \phi_{a}'=0$. From these calculations, we have the following corollary.
\begin{corollary}
	The induced structure $(\phi_{a},U_{a},u_{a},U,u,V,v)$, for each  $a=1,2,3$, on $(M,g)$ is not unique.
\end{corollary}
\noindent Now, let $\lambda_{a}$ be an eigenvalue of $\phi_{a}$, with respect to eigenvector $\zeta_{a}$, for all $a=1,2,3$. Then, $\phi_{a}\zeta_{a}=\lambda_{a}\zeta_{a}$. Applying $\phi_{a}$ to the previous relation and then using Proposition \ref{prop1}, we get $\lambda_{a}^{2}\zeta_{a}=\phi_{a}^{2}\zeta_{a}=-\zeta_{a}+u_{a}(\zeta_{a})U_{a}+u(\zeta_{a})U+v(\zeta_{a})V$. Applying $\phi_{a}$ to this relation gives $(\lambda_{a}^{2}+1)\lambda_{a}=0$. Therefore, we have the following corollary for each $\phi_{a}$. 
\begin{corollary}
 The eigenvalues of $\phi_{a}$, for all $a=1,2,3$, are $0$, $-i$ and  $i$. 
\end{corollary}
\noindent Furthermore, in view of Proposition \ref{prop1} and the definition of an almost contact $3$-structure (see \cite[p. 325]{ku}), we have the following result.
\begin{corollary}
Under the assumptions of Proposition \ref{prop1}, with $\xi$ and $N$  globally defined on $M$, $({\phi_{a}}_{|\mathcal{V}^{\perp}},U_{a},u_{a})$ defines an almost contact $3$-structure on the complement $\mathcal{V}^{\perp}$ of the vertical distribution in $TM$.
\end{corollary}
\begin{lemma}\label{lemma2}
Let $(M,g)$ be a null hypersurface of an indefinite complex contact manifold $\overline{M}$. Then, we have 
\begin{align}
	B(X,U)&=-g(X,V_{2})-\overline{g}(h_{U}X,V_{2}),\label{e1} \\
	C(X,U)&=-g(X,U_{2})-\overline{g}(h_{U}X,U_{2}),\label{e2}\\
	B(X,V)&=-g(X,V_{3})-\overline{g}(h_{V}X,V_{3}),\label{e3}\\
	C(X,V)&=-g(X,U_{3})-\overline{g}(h_{V}X,U_{3}),\label{e4}\\
	\mbox{and}\;\;\;\sigma(X)&=g(\nabla_{X}U,V),\;\;\;\forall\, X\in \Gamma(TM).\label{e5}
\end{align}
\end{lemma}
\begin{proof}
	In view of (\ref{cm4}) and first relation of (\ref{flow6}), we have 
	\begin{align}\label{e6}
		\nabla_{X}U+B(X,U)N=-GX-Gh_{U}X+\sigma(X)V,
	\end{align}
	for any $X\in \Gamma(TM)$. The inner product of (\ref{e6}) with $\xi$ and $N$ in turns gives 
	\begin{align}
		B(X,U)&=-\overline{g}(GX,\xi)-\overline{g}(Gh_{U}X,\xi)\label{e7}\\
	\mbox{and}\;\;\;\overline{g}(\nabla_{X}U,N)&=-\overline{g}(GX,N)-\overline{g}(Gh_{U}X,N).\label{e8}
	\end{align}
	Then, applying (\ref{cm2}), we get $B(X,U)=\overline{g}(X,G\xi)+\overline{g}(h_{U}X,G\xi)$, which implies (\ref{e1}). On the other hand, (\ref{cm2}), (\ref{flow7}) and (\ref{e8}) gives $C(X,U)=\overline{g}(X,GN)+\overline{g}(h_{U}X,GN)$, which proves (\ref{e2}). Relations (\ref{e3}) and (\ref{e4}) follows easily as in (\ref{e1}) and (\ref{e2}), while considering (\ref{cm5}), (\ref{flow6}), (\ref{flow7}) and (\ref{cm2}). Finally, (\ref{e5}) follows from (\ref{e6}), (\ref{cm2}) and (\ref{cm3}), which completes the proof.
\end{proof}
A null hypersurface $(M,g)$ of a semi-Riemannian manifold $(\overline{M},\overline{g})$ is called; \textit{totally umbilic} \cite[p. 106]{db} if and only if, on each coordinate neighbourhood $\mathcal{U}$ of $M$  there exist a smooth function $\rho$ such that $A_{\xi}^{*}X=\rho PX$, or equivalently, $B(X,PY)=\rho g(X,Y)$, for all $X,Y\in \Gamma(TM)$. In case $\rho=0$, we say that $M$ is \textit{totally geodesic} otherwise it is  \textit{proper} totally umbilic. In the same line, $M$ is \textit{screen totally umbilic}  \cite[p. 109]{db} if and only if, on each coordinate neighbourhood $\mathcal{U}$ of $M$ there exist a smooth function $\varrho$  such that $A_{N}X=\varrho PX$, or equivalently, $C(X,PY)=\varrho g(X,Y)$, for all $X,Y\in \Gamma(TM)$. In case $\varrho=0$,  we say that $M$ is \textit{screen totally geodesic} otherwise it is \textit{proper} screen totally umbilic. Furthermore, $M$ is \textit {screen locally conformal} \cite[p. 179]{ds2} if and only if, on any coordinate neighbourhood $\mathcal{U}$ there exist a non-vanishing smooth function $\psi$ such that $A_{N}X=\psi A_{\xi}^{*}X$, or equivalently, $C(X,PY)=\psi B(X,Y)$, for any $X,Y\in \Gamma(TM)$. The conformality is said to be global if $\mathcal{U}=M$. In the sequel, by screen conformal we shall mean screen locally conformal.

Next, in view of Lemma \ref{lemma2}, we have the following characterization result.
\begin{theorem}\label{main2}
	A normal indefinite complex contact manifold $\overline{M}$ does not admit any totally umbilic, screen totally umbilic or screen conformal null hypersurface $(M,g)$, tangent to the characteristic subbundle $\mathcal{V}$.
\end{theorem}
\begin{proof}
	As $\overline{M}$ is normal, we have $h_{U}=0$, for any $U\in \Gamma(\mathcal{V})$. Now, assume that $M$ is totally umbilic, then (\ref{e1}) and (\ref{e3}) implies $\rho g(X,U)=-g(X,V_{2})$ and $\rho g(X,V)=-g(X,V_{3})$, for all $X\in \Gamma(TM)$ and $U,V\in \Gamma(\mathcal{V})$. Setting $X=U_{2}$ and $X=U_{3}$ in the first and second of the previous relations and noting that $U_{2},U_{3}\perp \mathcal{V}$ (see decomposition \ref{e10}), we, respectively, get $-g(U_{2},V_{2})=0$ and $-g(U_{3},V_{3})=0$, which are both contradictions. Therefore, $M$ is never totally umbilic. On the other hand, if $M$ is screen totally geodesic, (\ref{e2}) and (\ref{e4}) leads to $\varrho g(X,U)=-g(X,U_{2})$ and $\varrho g(X,V)=-g(X,U_{3})$, for all $X\in \Gamma(TM)$. Letting $X=V_{2}$ and $X=V_{3}$ in the first and second relations gives $-g(V_{2},U_{2}=0$ and $-g(V_{3},U_{3})=0$, which are both contradictions. Hence, $M$ is never totally umbilic in $\overline{M}$. Finally, assume that $M$ is screen conformal, then (\ref{e10}) and (\ref{e2}) leads to $\psi g(X,V_{2})=g(X,U_{2})$, while (\ref{e3}) and (\ref{e4})  gives $\psi g(X,V_{3})=g(X,U_{3})$, for all $X\in \Gamma(TM)$. Setting $X=V_{2}$ in the first one and $X=V_{3}$ in the second, while noting that $V_{2}$ and $V_{3}$ are both null vector fields, gives $g(V_{2},U_{2})=0$ and $g(V_{3},U_{3})=0$. These are contradictions, and hence $M$ is never screen conformal in $\overline{M}$, which completes the proof.
\end{proof}
According to \cite[p. 89]{db}, $S(TM)$ is parallel with respect to $\nabla$ if $\nabla_{X}PY\in \Gamma(S(TM))$, for all $X,Y\in \Gamma(TM)$. It then follows from (\ref{flow7}) that $C(X,PY)=0$, for all $X,Y\in \Gamma(TM)$. That is, $S(TM)$ is totally geodesic. Furthermore, we see from (\ref{p40}) that the induced connection is a metric connection if and only if $B=0$, i.e., $M$ is totally geodesic. In fact, assume that $\nabla$ is a metric connection, then (\ref{p40}) implies that $B(X,Z)\eta(Y)+B(X,Y)\eta(Z)=0$, for all $X,Y,Z\in \Gamma(TM)$, since $\nabla g=0$. Hence, setting $Z=\xi$ in this relation and using the fact that $B(X,\xi)=0$, for all $X\in \Gamma(TM)$ and $\xi \in \Gamma(TM^{\perp})$, we get $B(X,Y)=0$, for all $X,Y\in \Gamma(TM)$. The converse is obvious. The normal distribution $TM^{\perp}$ is said to be killing if $\pounds_{\xi}g=0$, for all $\xi\in \Gamma(TM^{\perp})$. Here, $\pounds_{\xi}$ denotes the usual Lie derivative with respect to $\xi$. By a simple calculation, while using (\ref{p40}) and the second relation of (\ref{flow7}), this is equivalent to $2B=0$, i.e. $M$ totally geodesic.

Putting all the above details to Theorem \ref{main2}, we have the following.
\begin{corollary}
	In view of Theorem \ref{main2}, we see that 
	\begin{enumerate}
		\item $S(TM)$ is never parallel,
		\item $\nabla$ is never a metric connection,
		\item $TM^{\perp}$ is never a killing distribution,
	\end{enumerate}
	on any null hypersurface of a normal indefinite complex contact manifold $\overline{M}$. 
\end{corollary}
\begin{lemma}\label{lemma3}
	For a null hypersurface $(M,g)$ of a normal indefinite complex contact manifold $\overline{M}$, the following holds
	\begin{align}
	C(X,V_{1})=B(X,&U_{1})-u(X)d\sigma(\xi,U_{2})-v(X)d\sigma(\xi,U_{3}),\label{h1}\\
		C(X,V_{2})&=B(X,U_{2})+v(X)d\sigma(V_{2},U_{2}),\label{e25}\\
		\mbox{and}\;\;\;\;\;C(X,&V_{3})=B(X,U_{3})-u(X)d\sigma(V_{3},U_{3}),\label{e26}
	\end{align}
	for all $X\in \Gamma(TM)$.
\end{lemma}
\begin{proof}
Setting $Y=N$ and $Z=\xi$ in (\ref{m1}) and applying Lemma \ref{lemma1}, we get 
\begin{align}\label{m2}
	\overline{g}((\overline{\nabla}_{X}J)N,\xi)&=u(X) d\sigma(\xi,GN)+v(X) d\sigma(\xi,HN) \nonumber\\
	&=-u(X) d\sigma(\xi,U_{2})-v(X) d\sigma(\xi,U_{3}), 
\end{align}
for all $X\in \Gamma(T\overline{M})$. With the help of (\ref{e11}), the left hand side of (\ref{m2}) gives 
\begin{align}
	\overline{g}((\overline{\nabla}_{X}J)N,\xi)&=\overline{g}(\overline{\nabla}_{X}JN,\xi)-\overline{g}(J\overline{\nabla}_{X}N,\xi)\nonumber\\
	&= -\overline{g}(\overline{\nabla}_{X}U_{1},\xi)-\overline{g}(\overline{\nabla}_{X}N,V_{1}).\label{m3}
\end{align}
Now, for all $X\in \Gamma(TM)$, (\ref{m3}) simplifies to 
\begin{align}\label{m4}
	\overline{g}((\overline{\nabla}_{X}J)N,\xi)=-B(X,U_{1})+C(X,V_{1}).
\end{align}
Thus, (\ref{h1}) follows from (\ref{m4}) and (\ref{m2}). On the other hand,  setting $Y=N$ and $Z=\xi$ in (\ref{e20}) and using Lemma \ref{lemma1}, we get 
\begin{align}\label{e22}
		\overline{g}((\overline{\nabla}_{X}G)N,\xi)=v(X)d\sigma(G\xi,GN)=v(X)d\sigma(V_{2},U_{2}),
	\end{align}
for all $X\in \Gamma(T\overline{M})$. Simplifying the left hand of (\ref{e22}), we see that 
\begin{align}
	\overline{g}((\overline{\nabla}_{X}G)N,\xi)&=\overline{g}(\overline{\nabla}_{X}GN,\xi)-\overline{g}(G\overline{\nabla}_{X}N,\xi)\nonumber\\
	&= -\overline{g}(\overline{\nabla}_{X}U_{2},\xi)-\overline{g}(\overline{\nabla}_{X}N,V_{2}),\label{e23}
\end{align}
for $X\in \Gamma(T\overline{M})$, in which we have used (\ref{cm2}). For all $X\in \Gamma(TM)$, (\ref{e23}), (\ref{flow6}) and (\ref{flow7}) gives 
\begin{align}\label{e24}
	\overline{g}((\overline{\nabla}_{X}G)N,\xi)=-B(X,U_{2})+C(X,V_{2}).
\end{align}
Then, (\ref{e25}) follows immediately from (\ref{e22}) and (\ref{e24}). Lastly, relation (\ref{e26}) follows by similar calculations, while using (\ref{e21}), (\ref{flow6}) and  (\ref{flow7}).
\end{proof}
With the aid of Lemma \ref{lemma3}, we have the following result.
\begin{theorem}
	The transversal bundle $\mathrm{tr}(TM)$ of a null hypersurface $(M,g)$, tangent to the characteristic subbundle $\mathcal{V}$, of a normal indefinite complex contact manifold $\overline{M}$ is never a killing distribution.
\end{theorem}
\begin{proof}
	Suppose that $\mathrm{tr}(TM)$ is a killing distribution. Then, $\pounds_{N}\overline{g}=0$, for all $N\in \Gamma(\mathrm{tr}(TM))$.  This implies that 
	\begin{align}\label{e27}
		\overline{g}(\overline{\nabla}_{X}N,Y)+\overline{g}(X,\overline{\nabla}_{Y}N)=0,
	\end{align}
	for all $X,Y\in \Gamma(T\overline{M})$. For all $X,Y\in \Gamma(TM)$, (\ref{e27}) and (\ref{flow6}) gives 
	\begin{align}\label{e28}
		C(X,PY)+C(Y,PX)=\tau(X)\eta(Y)+\tau(Y)\eta(X).
	\end{align}
	Setting $Y=U$ in (\ref{e28}) and using (\ref{e2}), we get $C(U,PX)=g(X,U_{2})+\tau(U)\eta(X)$. And, for $X=V_{2}$ in the previous relation we have $C(U,V_{2})=1$. On the other hand, (\ref{e25}) and (\ref{e1}) implies that $C(U,V_{2})=B(U,U_{2})=B(U_{2},U)=-g(U_{2},V_{2})=-1$, which contradicts the previous relation. Hence, the transversal bundle is never a killing a distribution.
\end{proof}
\section{Main results}\label{mainn}
We have seen  (see Theorem \ref{main2}) that an indefinite complex contact manifold does not admit any totally umbilic, screen totally umbilic or screen conformal null hypersurface, tangent to the characteristic subbundle $\mathcal{V}$. It is easy to see that the well-known definitions of the above mentioned null hypersurfaces fails in portions of $TM$ which includes the vector fields $V,U$  spanning  $\mathcal{V}$. However, these definitions can be confined to $\mathcal{V}^{\perp}:=TM^{\perp}\perp\{D_{1}\oplus D_{2}\}\perp D_{0}$, which is the complementary distribution of $\mathcal{V}$ in $TM$ (see decomposition (\ref{00}) of $TM$). Such considerations give rise to totally contact umbilic, totally contact screen umbilic and contact screen conformal null hypersurfaces (see \cite{ssekajja}) in case the ambient manifold is an indefinite Sasakian manifold. In the same way, if we let $\tilde{P}$ be the projection morphism of $TM$ onto $TM^{\perp}\perp\{D_{1}\oplus D_{2}\}\perp D_{0}$, then each $X\in \Gamma(TM)$ can be written as 
\begin{align}\label{le3}
	X=\tilde{P}X+u(X)U+v(X)V,
\end{align}
where $u(X)=g(U,X)$ and $v(X)=g(V,X)$. Hence, we have the following definition.
\begin{definition}\label{definition2}
\rm {
	Let $(M,g)$ be a null hypersurface, tangent to the characteristic subbundle $\mathcal{V}$, of an indefinite complex contact manifold $\overline{M}$. Then, 
	\begin{enumerate}
		\item $M$ is totally contact umbilic if and only if on each coordinate neighbourhood $\mathcal{U}|_{\mathcal{V}^{\perp}}$ of $M$, there exists a smooth function $\beta$ such that $B=\beta \otimes g$ on $\mathcal{V}^{\perp}$ or equivalently
		\begin{align}\label{v1}
			B(\tilde{P}X,\tilde{P}Y)=\beta g(\tilde{P}X,\tilde{P}Y),\;\;\;\forall \, X,Y\in \Gamma(TM).
		\end{align}
		\item $M$ is totally contact screen umbilic if and only if on each coordinate neighbourhood $\mathcal{U}|_{\mathcal{V}^{\perp}}$ of $M$, there exists a smooth function $\mu$ such that $B=\beta \otimes g$ on $\mathcal{V}^{\perp}$ or equivalently
	\begin{align}\label{v2}
			C(\tilde{P}X,P\tilde{P}Y)=\mu g(\tilde{P}X,\tilde{P}Y),\;\;\;\forall \, X,Y\in \Gamma(TM). 
		\end{align}
		In case $\mu=0$, then $M$ is called totally contact screen geodesic.
	\item $M$ is contact screen conformal if and only if on each coordinate neighbourhood $\mathcal{U}|_{\mathcal{V}^{\perp}}$ of $M$, there exists a nonzero smooth function $\varphi$  such that $C=\varphi \otimes B$ on $\mathcal{V}^{\perp}$ or equivalently			
			\begin{align}\label{v3}
			C(\tilde{P}X,P\tilde{P}Y)=\varphi B(\tilde{P}X,\tilde{P}Y),\;\;\;\forall \, X,Y\in \Gamma(TM), 
		 \end{align}
		 and $M$ is contact screen homothetic if $\varphi$  is a constant function.
	\end{enumerate}
	}
\end{definition} 
\noindent By direct calculations using (\ref{m1}), (\ref{e20}), (\ref{e21}), (\ref{flow6}) and (\ref{flow7}), we have 
\begin{align}
	g((\nabla_{X}\phi_{1})Y,V_{1})&=-B(X,Y)+u_{1}(Y)C(X,V_{1}),\label{q1}\\
	g((\nabla_{X}\phi_{2})Y,V_{2})&=-B(X,Y)+u_{2}(Y)C(X,V_{2})+\sigma(X)u_{1}(Y),\label{q2}\\
	g((\nabla_{X}\phi_{3})Y,V_{3})&=-B(X,Y)+u_{3}(Y)C(X,V_{3})+\sigma(X)u_{1}(Y),\label{q3}
\end{align}
for all $X,Y\in \Gamma(\mathcal{V}^{\perp})$. Moreover, $u_{1}$, $u_{2}$ and $u_{3}$ satisfies the following relations
\begin{align}
	(\nabla_{X}u_{1})(Y)&=-B(X,\phi_{1}Y)-u_{1}(Y)\tau(X),\\
	(\nabla_{X}u_{2})(Y)&=-B(X,\phi_{2}Y)-u_{2}(Y)\tau(X)+\sigma(X)u_{3}(Y),\\
	(\nabla_{X}u_{3})(Y)&=-B(X,\phi_{3}Y)-u_{3}(Y)\tau(X)-\sigma(X)u_{2}(Y),
\end{align}
for all $X,Y\in \Gamma(\mathcal{V}^{\perp})$. It then follows from (\ref{q1}), (\ref{q2}) and (\ref{q3}) that $B(X,U_{a})=C(X,V_{a})$ and $g((\nabla_{X}\phi_{a})Y,V_{a})=-B(X,Y)$, for all $X\in \Gamma(\mathcal{V}^{\perp})$ and $Y\in \Gamma(D)$. Therefore, we have the following result.
\begin{theorem}
	Let $\overline{M}$ be a normal indefinite  complex contact manifold, and $(M,g)$ a null hypersurface of $\overline{M}$, tangent to the characteristic subbundle $\mathcal{V}$. Then, 
	\begin{enumerate}
		\item $M$ is totally contact geodesic if and only if $C(X,V_{a})=0$ and $(\nabla_{X}\phi_{a})Y=0$, for all $X\in \Gamma(\mathcal{V}^{\perp})$ and $Y\in \Gamma(D)$.
		\item $M$ is totally contact screen geodesic if and only if $B(X,U_{a})=0$ and $\nabla_{X}PY\notin \Gamma(TM^{\perp})$, for all $X\in \Gamma(\mathcal{V}^{\perp})$ and $Y\in \Gamma(D_{2}\perp D_{0})$.	
\end{enumerate}
\end{theorem}
In view of (\ref{cm6}) and (\ref{v34}), we have 
\begin{align}\label{q6}
	&\;\frac{c+3}{4}[\overline{g}(Z,JY)u_{1}(X)-\overline{g}(Z,JX)u_{1}(Y)+2\overline{g}(X,JY)u_{1}(Z)]\nonumber\\
	&+\frac{c-1}{4}[\overline{g}(Z,GY)u_{2}(X)-\overline{g}(Z,GX)u_{2}(Y)+2\overline{g}(X,GY)u_{2}(Z)\nonumber\\
	&+ \overline{g}(Z,HY)u_{3}(X)-\overline{g}(Z,HX)u_{3}(Y)+2\overline{g}(X,HY)u_{3}(Z)]\nonumber\\
	&=(\nabla_{X}B)(Y,Z)-(\nabla_{Y}B)(X,Z)+\tau(X)B(Y,Z)-\tau(Y)B(X,Z),
\end{align}
for all $X,Y,Z\in \Gamma(\mathcal{V}^{\perp})$. On the other hand, using (\ref{cm6}) and (\ref{v35}), we get 
\begin{align}\label{q7}
	\frac{c+3}{4}&[\overline{g}(Y,Z)\eta(X)-\overline{g}(X,Z)\eta(Y)+\overline{g}(Z,JY)v_{1}(X)\nonumber\\
	&-\overline{g}(Z,JX)v_{1}(Y)	+2\overline{g}(X,JY)v_{1}(Z)]+\frac{c-1}{4}[\overline{g}(Z,GY)v_{2}(X)\nonumber\\
	&-\overline{g}(Z,GX)v_{2}(Y)+2\overline{g}(X,GY)v_{2}(Z)+ \overline{g}(Z,HY)v_{3}(X)\nonumber\\
	&-\overline{g}(Z,HX)v_{3}(Y)+2\overline{g}(X,HY)v_{3}(Z)]=(\nabla_{X}C)(Y,PZ)\nonumber\\
	&-(\nabla_{Y}C)(X,PZ)-\tau(X)C(Y,PZ)+\tau(Y)C(X,PZ), 
\end{align}
for all $X,Y,Z\in \Gamma(\mathcal{V}^{\perp})$.
\begin{lemma}\label{lemma5}
	Let $(M,g)$ be a null hypersurface, tangent to the characteristic subbundle $\mathcal{V}$, of a normal indefinite complex contact manifold  $\overline{M}$. Then, if 
	\begin{enumerate}
		\item $M$ is totally contact umbilic, we have
		\begin{align}\label{le1}
			&(\nabla_{X}B)(Y,Z)-(\nabla_{Y}B)(X,Z)=(X\beta)g(Y,Z)-(Y\beta)g(X,Z)	\nonumber\\
			&\;\;\;\;\;\;\;\;\;\;\;\;+\beta^{2}[g(X,Z)\eta(Y)-g(Y,Z)\eta(X)	]+u_{2}(Z)\overline{g}(GX,Y)\nonumber\\
			&\;\;\;\;\;\;\;\;\;\;\;\;-u_{2}(Z)\overline{g}(GY,X)+u_{3}(Z)\overline{g}(HX,Y)-u_{3}(Z)\overline{g}(HY,X)\nonumber\\
			&\;\;\;\;\;\;\;\;\;\;\;\;+u_{2}(Y)\overline{g}(GX,Z)-u_{2}(X)\overline{g}(GY,Z)+u_{3}(Y)\overline{g}(HX,Z)\nonumber\\
			&\;\;\;\;\;\;\;\;\;\;\;\;-u_{3}(X)\overline{g}(HY,Z),
	\end{align}
	\item $M$ is totally contact screen umbilic, we have
	   \begin{align}\label{le2}
			&(\nabla_{X}C)(Y,PZ)-(\nabla_{Y}C)(X,PZ)=(X	\mu)g(Y,PZ)-(Y\mu)g(X,PZ)\nonumber\\
			&\;\;\;\;\;\;\;\;\;\;+\mu [B(X,Z)\eta(Y)-B(Y,Z)\eta(X)]-C(U,PZ)[\overline{g}(GX,Y)\nonumber\\
			&\;\;\;\;\;\;\;\;\;\;-\overline{g}(GY,X)]-C(V,PZ)[\overline{g}(HX,Y)-\overline{g}(HY,X)]+v_{2}(Y)\overline{g}(GX,Z)\nonumber\\
			&\;\;\;\;\;\;\;\;\;\;-v_{2}(X)\overline{g}(GY,Z)+v_{3}(Y)\overline{g}(HX,Z)-v_{3}(X)\overline{g}(HY,Z),
	\end{align}
   \end{enumerate}
   for all $X,Y,Z\in \Gamma(\mathcal{V}^{\perp})$.
\end{lemma}
\begin{proof}
	Using (\ref{p100}), (\ref{le3}), (\ref{v1}) and Lemma \ref{lemma2}, we derive
	\begin{align}\label{le6}	
		(\nabla_{X}B)(Y,Z)&=X(B(Y,Z))-B(\tilde{P}\nabla_{X}Y,Z)-B(Y,\tilde{P}\nabla_{X}Z)\nonumber\\
		&\;\;\;\;+u(\nabla_{X}Y)u_{2}(Z)+v(\nabla_{X}Y)u_{3}(Z)+	u(\nabla_{X}Z)u_{2}(Y)\nonumber\\
		&\;\;\;\;+v(\nabla_{X}Z)u_{3}(Y)\nonumber\\
		&= (X\beta)g(Y,Z)+\beta 	[Xg(Y,Z)-g(\nabla_{X}Y,Z)-g(Y,\nabla_{X}Z)]	\nonumber\\
		&\;\;\;\;+u(\nabla_{X}Y)u_{2}(Z)+v(\nabla_{X}Y)u_{3}(Z)+	u(\nabla_{X}Z)u_{2}(Y)\nonumber\\
		&\;\;\;\;+v(\nabla_{X}Z)u_{3}(Y),
\end{align}
 for all $X,Y,Z\in \Gamma(\mathcal{V}^{\perp})$.  On the other hand, using (\ref{p40}), (\ref{cm4}) and (\ref{flow6}), we derive
 \begin{align}\label{le7}
 	u(\nabla_{X}Y)=g(\nabla_{X}Y,U)=-(\nabla_{X}g)(Y,U)-g(Y,\nabla_{X}U)=\overline{g}(GX,Y).
 \end{align}
 In the same way, we have 
  \begin{align}\label{le8}
 	v(\nabla_{X}Y)=g(\nabla_{X}Y,V)=-(\nabla_{X}g)(Y,V)-g(Y,\nabla_{X}V)=\overline{g}(HX,Y).
 \end{align} 
 In view of (\ref{le6}), (\ref{le7}), (\ref{le8}), (\ref{p40}) and (\ref{v1}),  we have 
 \begin{align}\label{le9}
 	(\nabla_{X}B)(Y,Z)&=(X\beta)g(Y,Z)+\beta^{2}[g(X,Z)\eta(Y)+g(X,Y)\eta(Z)]\nonumber\\
 	&\;\;\;\;\; +u_{2}(Z)\overline{g}(GX,Y)+u_{3}(Z)\overline{g}(HX,Y)+u_{2}(Y)\overline{g}(GX,Z)\nonumber\\
 	&\;\;\;\;\;+u_{3}(Y)\overline{g}(HX,Z),
  \end{align}
  for all $X,Y,Z\in \Gamma(\mathcal{V}^{\perp})$. Then, relation (\ref{le1})  follows from (\ref{le9}). Furthermore, using (\ref{p101}), (\ref{le3}), (\ref{v2}), (\ref{cm4}), (\ref{cm5}) and the relation $X=PX+\eta(X)\xi$, for all $X\in \Gamma(TM)$, we derive 
  \begin{align}\label{le10}
  	(\nabla_{X}C)(Y,PZ)&=(X\mu)g(Y,PZ)+\mu[(\nabla_{X}g)(Y,Z)-g(\nabla_{X}\xi,Y)\eta(Z)]\nonumber\\
  	&\;\;\;\;-C(U,PZ)\overline{g}(GX,Y)-C(V,PZ)\overline{g}(HX,Y)\nonumber\\
  	&\;\;\;\;+v_{2}(Y)\overline{g}(GX,Z)+V_{3}(Y)\overline{g}(HX,PZ),
  \end{align}
 for all $X,Y,Z\in \Gamma(\mathcal{V}^{\perp})$. Then, applying (\ref{flow7}) and (\ref{p40}) to (\ref{le10}), we get 
 \begin{align}\label{le11}
  	(\nabla_{X}C)(Y,PZ)&=(X\mu)g(Y,PZ)+\mu[B(Y,Z)\eta(X)+2B(X,Y)\eta(Z)]\nonumber\\
  	&\;\;\;\;-C(U,PZ)\overline{g}(GX,Y)-C(V,PZ)\overline{g}(HX,Y)\nonumber\\
  	&\;\;\;\;+v_{2}(Y)\overline{g}(GX,Z)+V_{3}(Y)\overline{g}(HX,PZ),
  \end{align}
 for all $X,Y,Z\in \Gamma(\mathcal{V}^{\perp})$. Finally, (\ref{le2}) follows from (\ref{le11}), which completes the proof.
 \end{proof}
\begin{theorem}\label{main4}
	Let $\overline{M}$ be a normal indefinite  complex contact manifold, and $(M,g)$ a totally contact umbilic or totally contact screen umbilic null hypersurface of $\overline{M}$, tangent to the characteristic subbundle $\mathcal{V}$. Then, $c=-3$, that is; $\overline{M}$ is a space of constant $GH$-sectional curvature $-3$.
\end{theorem}
\begin{proof}

	Assume that $M$ is totally contact umbilic. Then, letting $X=\xi$ in (\ref{le1}) of Lemma \ref{lemma5}, we get 
	\begin{align}\label{le12}
		(\nabla_{\xi}B)(Y,Z)-(\nabla_{Y}B)(\xi,Z)&=[\xi \beta-\beta]g(Y,Z)\nonumber\\
		&-3u_{2}(Y)u_{2}(Z)-3u_{3}(Y)u_{3}(Z),
	\end{align}
	for all $X,Y,Z\in \Gamma(\mathcal{V}^{\perp})$.	Then letting $X=\xi$ in (\ref{q6}) and then using (\ref{le12}), we get 
	\begin{align}\label{le13}
		[\xi \beta+\beta \tau(\xi)-\beta^{2}]g(Y,Z)=\frac{3}{4}(c+3)\sum_{a=1}^{3}u_{a}(Y)u_{a}(Z),
	\end{align}
	for all $X,Y,Z\in \Gamma(\mathcal{V}^{\perp})$. Letting $Y=Z=U_{a}$, $a=1,2,3$, we get $c=-3$. Moreover, $\xi \beta+\beta \tau(\xi)-\beta^{2}=0$. On the other hand, if $M$ is totally contact screen umbilic, we let $X=\xi$	 in (\ref{le2}) and get 
	\begin{align}\label{le14}
		(\nabla_{\xi}C)(Y,PZ)&-(\nabla_{Y}C)(\xi,PZ)=(\xi \mu)g(Y,PZ)-\mu B(Y,Z)\nonumber\\
		&\;\;+2u_{2}(Y)C(U,PZ)+2u_{3}(Y)C(V,PZ)\nonumber\\
		&\;\;\;\;-v_{2}(Y)u_{2}(Z)-v_{3}(Y)u_{3}(PZ),
	\end{align}
	for all $X,Y,Z\in \Gamma(\mathcal{V}^{\perp})$. Using (\ref{le14}) in (\ref{q7}), with $X=\xi$, we get 
	\begin{align}\label{le15}
		&[\xi \mu -\mu \tau(\xi)]g(Y,PZ)-\mu B(Y,Z)=-2u_{2}(Y)C(U,PZ)-2u_{3}(Y)C(V,PZ)\nonumber\\
		&\;\;\;\;\;\;\;\;\;\;\;+\frac{c+3}{4}[g(Y,Z)-u_{1}(Z)v_{1}(Y)+2u_{1}(Y)v_{1}(Z)+v_{2}(Y)u_{2}(Z)\nonumber\\
		&\;\;\;\;\;\;\;\;\;\;\;\;\;\;\;\;\;\;\;+v_{3}(Y)u_{3}(Z)]+\frac{c-1}{2}[u_{2}(Y)v_{2}(Z)+u_{3}(Y)v_{3}(Z)],
	\end{align}
	for all $X,Y,Z\in \Gamma(\mathcal{V}^{\perp})$. Letting $Y=V_{a}$ and $Z=U_{a}$, $a=1,2,3$, in (\ref{le15}) and noting, from Lemma \ref{lemma3} and (\ref{v2}),  that $B(V_{a},U_{a})=C(V_{a},V_{a})=0$, we get 
	\begin{align}\label{le16}
		\xi \mu-\mu \tau(\xi)=\frac{1}{2}(c+3).
	\end{align}
	On the other hand, using Lemmas \ref{lemma2} and \ref{lemma3}, we get $C(U,V_{2})=B(U,U_{2})=B(U_{2},U)=-1$ and $C(V,V_{3})=B(V,U_{3})=B(U_{3},V)=-1$. Thus, letting $Y=U_{a}$ and $Z=V_{a}$, $a=1,2,3$, in (\ref{le15}), we get 
\begin{align}\label{le17}
		\xi \mu-\mu \tau(\xi)=\frac{3}{4}(c+3).
	\end{align}
	Therefore, from (\ref{le16}) and (\ref{le17}), we get $c=-3$ and $\xi \mu-\mu \tau(\xi)=0$, which completes the proof.
\end{proof}
\noindent From Theorem \ref{main4}, we have the following.
\begin{corollary}
	A normal indefinite complex contact manifold $\overline{M}$ with $c\ne -3$ does not admit any totally contact umbilic or totally contact screen umbilic null hypersurface $(M,g)$, tangent to the characteristic subbundle $\mathcal{V}$. 
\end{corollary}
\noindent In view of Lemma \ref{lemma3} and Theorem \ref{main4}, we also have the following.
\begin{corollary}
	Any null hypersurface of $\overline{M}$ which is both totally umbilic and contact screen umbilic is totally contact geodesic and contact screen geodesic, that is $\beta=\mu =0$.
\end{corollary}
Unlike $B$, the local second fundamental form $C$ is not, generally, symmetric on $S(TM)$. In fact, by a direct calculation, we have $C(X,Y)-C(Y,X)=\eta([X,Y])$, for all $X,Y\in \Gamma(S(TM))$. It the follows from the above relation that $C$ is symmetric on $S(TM)$ if and only if $S(TM)$ is integrable. Now, assume that $S(TM)$ is integrable, then 
\begin{align}
	&C(U,PZ)=C(PZ,U)=-v_{2}(PZ),\label{n1}\\
	\mbox{and}\;\;\;\;\;& C(V,PZ)=C(PZ,V)=-v_{3}(PZ),\label{n2}
\end{align}
for any $Z\in \Gamma(TM)$, in which we have used Lemma \ref{lemma2}. On the other hand, for a totally contact screen umbilic null hypersurface $M$, we have seen that $c=-3$ and $\xi \mu-\mu \tau(\xi)=0$. Considering these relations in (\ref{le15}), together with (\ref{n1}) and (\ref{n2}), we get 
\begin{align}\label{n3}
	-\mu B(Y,PZ)=0,\;\;\;  \forall\,Y,Z\in \Gamma(\mathcal{V}^{\perp}). 
\end{align}
Setting $Y=Z=U_{2}$ in (\ref{n3}), and then apply Lemma \ref{lemma3} and (\ref{v2}), we get $0=-\mu B(U_{2},U_{2})=-\mu C(U_{2},V_{2})=-\mu^{2}$. Hence, from Theorem \ref{main4}, we have the following corollary.
\begin{corollary}
	Under the assumptions of Theorem \ref{main4}, any screen integrable totally contact screen umbilic nullhypersurface of $\overline{M}$ is totally contact screen geodesic, that is $\mu=0$.
\end{corollary}
By direct calculations, while using (\ref{p100}), (\ref{p101}),(\ref{cm4}), (\ref{cm5}), (\ref{p40}), (\ref{le3}), (\ref{v3}) and Lemma \ref{lemma2}, we note that on a contact screen conformal null hypersurface, tangent to $\mathcal{V}$, of a normal indefinite complex contact manifold $\overline{M}$, we have 
	\begin{align}\label{b1}
		&(\nabla_{X}C)(Y,PZ)-(\nabla_{Y}C)(X,PZ)-\varphi [(\nabla_{X}B)(Y,PZ)-(\nabla_{Y}B)(X,PZ)]\nonumber\\
		&\;\;\;\;\;\;\;\;\;\;\;\;\;\;=(X\varphi) B(Y,PZ)-(Y\varphi)B(X,PZ)+\varphi [u_{2}(X)\overline{g}(GY,PZ)\nonumber\\
		&\;\;\;\;\;\;\;\;\;\;\;\;\;\;-u_{2}(Y)\overline{g}(GX,PZ)+u_{3}(X)\overline{g}(HY,PZ)-u_{3}(Y)\overline{g}(HX,PZ)]\nonumber\\
		&\;\;\;\;\;\;\;\;\;\;\;\;\;\;+C(U,PZ)[\overline{g}(GY,X)-\overline{g}(GX,Y)]+C(V,PZ)[\overline{g}(HY,X)\nonumber\\
		&\;\;\;\;\;\;\;\;\;\;\;\;\;\;-\overline{g}(HX,Y)]+v_{2}(Y)\overline{g}(GX,PZ)-v_{2}(X)\overline{g}(GY,PZ)\nonumber\\
		&\;\;\;\;\;\;\;\;\;\;\;\;\;\;+v_{3}(Y)\overline{g}(HX,PZ)-v_{3}(X)\overline{g}(HY,PZ),
		\end{align}
for all $X,Y,Z\in \Gamma(\mathcal{V}^{\perp})$.

\begin{theorem}\label{main5}
	Let $\overline{M}$ be a normal indefinite  complex contact manifold, and $(M,g)$ a contact screen conformal null hypersurface of $\overline{M}$, tangent to the characteristic subbundle $\mathcal{V}$. Then, $c=-3$, i.e; $\overline{M}$ is a space of constant $GH$-sectional curvature $-3$.
\end{theorem}
\begin{proof}
	Setting $X=\xi$ in (\ref{b1}) and then apply (\ref{cm6}), (\ref{v34}) and (\ref{v35}), we derive
	\begin{align}\label{le20}
		&\;\;\;\;\;\;\;\;\;\;\;\;\;\;\frac{c+3}{4}[g(Y,PZ)+u_{1}(Z)v_{1}(Y)+2u_{1}(Y)v_{1}(Z)]\nonumber\\
		&\frac{c-1}{4}[u_{2}(Z)v_{2}(Y)+2u_{2}(Y)v_{2}(Z)+u_{3}(Z)v_{3}(Y)+2u_{3}(Y)v_{3}(Z)]\nonumber\\
		&\;\;\;-\varphi [\frac{3c+9}{4}u_{1}(Z)u_{1}(Y)+\frac{3c+1}{4}\{u_{2}(Y)u_{2}(Z)+u_{3}(Y)u_{3}(Z)\}\nonumber\\
		&\;\;\;\;\;\;\;\;\;\;-2\tau(\xi)B(Y,Z)]=(\xi \varphi)B(Y,Z)+2u_{2}(Y)C(U,PZ)\nonumber\\
		&\;\;\;\;\;\;\;\;\;\;\;\;\;\;\;\;\;\;\;\;\;\;\;\;+2u_{3}(Y)C(V,PZ)-v_{2}(Y)u_{2}(Z)\nonumber\\
		&\;\;\;\;\;\;\;\;\;\;\;\;\;\;\;\;\;\;\;\;\;\;\;\;\;\;\;\;\;\;\;\;\;\;\;\;\;-v_{3}(Y)u_{3}(Z),
	\end{align}
	for all $X,Y,Z\in \Gamma(\mathcal{V}^{\perp})$. Since, by Lemmas \ref{lemma2} and \ref{lemma3}, we have $C(U,V_{2})=B(U,U_{2})=B(U_{2},U)=-1$ and $C(V,V_{3})=B(V,U_{3})=B(U_{3},V)=-1$. Then, letting $Y=U_{a}$ and $Z=V_{a}$, $a=1,2,3$, in (\ref{le20}), we get 
\begin{align}\label{le21}
	[\xi \varphi-2\varphi \tau(\xi)]B(U_{a},V_{a})=\frac{3}{4}(c+3).
\end{align}
On the other hand,  letting $Y=V_{a}$ and $Z=U_{a}$, $a=1,2,3$, in (\ref{le20}), we get 
\begin{align}\label{le22}
	[\xi \varphi-2\varphi \tau{\xi}]B(V_{a},U_{a})=\frac{1}{2}(c+3).
\end{align}
It then follows from (\ref{le21}), (\ref{le22}) and the symmetry of $B$ that $c=-3$, and $[\xi \varphi-2\varphi \tau(\xi)]B(V_{a},U_{a})=0$, for all $a=1,2,3$, which completes the proof.
\end{proof}
\noindent The following is an immediate consequence of Theorem \ref{main5}.
\begin{corollary}
	There exist no any contact screen conformal null hypersurface of $\overline{M}$ such that $c\ne -3$.
\end{corollary}


\end{document}